\documentclass[thmsa,11pt,a4paper]{amsart}
\usepackage[a4paper,left=2.5cm,right=2.5cm,top=2.5cm,bottom=2.5cm]{geometry}%
\usepackage{amsfonts,amsmath,amssymb,amsthm,color}
\usepackage{graphicx}

\newtheorem{theorem}{Theorem}
\newtheorem{lemma}{Lemma}
\newtheorem{corollary}{Corollary}
\newtheorem{proposition}{Proposition}

\newcommand{\R}{{\mathbb R}}
\newcommand{\N}{{\mathbb N}}

\newcommand{\cN}{{\mathcal N}}

\newcommand{\cR}{{\mathcal R}}

\DeclareMathOperator*{\argmin}{arg\,min}
\DeclareMathOperator*{\argmax}{arg\,max}

\newcommand{\rank}{\mathrm{rank}}

\title[Greedy Gauss-Newton algorithm]{Greedy Gauss-Newton algorithm for finding sparse solutions to nonlinear underdetermined systems of equations}

\author[M.~Gulliksson]{M{\aa}rten Gulliksson}
\address{M.~Gulliksson, School of Science and Technology, \"{O}rebro University, Sweden }
\email{marten.gulliksson@oru.se}

\author[A.~Oleynik]{Anna Oleynik}
\address{A.~Oleynik, Department of Mathematical Sciences and Technology, Norwegian University of Life Sciences, Postboks $5003$ NMBU 1432 {\AA}s}
\email{anna.oleynik@nmbu.no}


\begin{document}
\maketitle
\begin{abstract}
We consider the problem of finding sparse solutions to a system of underdetermined
nonlinear system of equations. The methods are based on a Gauss-Newton approach with line search
where the search direction is found by solving a linearized problem using only a subset of the columns
in the Jacobian. The choice of columns in the Jacobian is made through a greedy approach looking at
either maximum descent or an approach corresponding to orthogonal matching for linear problems.
The methods are shown to be convergent and efficient and outperform the $\ell_1$ approach on the test
problems presented.
\end{abstract}

\section{Introduction}

We consider the nonlinear underdetermined system of equations
\begin{equation*}
\begin{array}{ccc}
    f_1(x_1,\ldots, x_N) & = & 0\\
		\vdots & &\\
		 f_m(x_1,\ldots, x_N) & = & 0
\end{array}
\end{equation*}
or simply
\begin{equation}
\label{eq:f=0}
  f(x)=0,
\end{equation}
where $x\in \R^N$ and $f: D \subset \R^N \rightarrow \R^m, m<N$ is twice continuously differentiable on the open convex set $D$, i.e., $f_i\in C^2(D), i=1, \ldots , m$. If $0\in f(D)$ the solution to \eqref{eq:f=0} is not unique, which is a direct consequence of the Implicit Function Theorem \cite{Apos}.
We refer to \cite{Ku,Ma89,TroppWright,SunZhengLi2013} for the examples from different application areas as motivation for solving (\ref{eq:f=0}).
In this paper we are interested in sparse solutions to \eqref{eq:f=0}, i.e., solutions that contain only a few nonzero components.
Let $\| x \|_0$ be the so-called $\ell_0$- norm (which is actually not a norm) on $\R^N$ defined as the number of nonzero elements
\[
  \| x \|_0 = \sharp \left\{ i : x_i \neq 0 \right\}.
\]
We say that a vector $x$ is $n$-sparse if $\|x\|_0 \leq n,$ and sparse if $\|x\|_0 \ll m.$

The problem of finding the most sparse solution to \eqref{eq:f=0} reads
\begin{equation}
\label{eq:problem:1}
\begin{array}{l}
    \min_x \| x \|_0\\
		\text{s.t. } f(x) = 0.
\end{array}
\end{equation}
Due to the combinatorial complexity  problem \eqref{eq:problem:1} is considered to be intractable, see \cite{TroppWright},
and  current algorithms can not guarantee that the (sparse) solution attained is a solution to \eqref{eq:problem:1}.

Linear problems, i.e., $f(x) = Ax-b, A \in \R^{m\times N}, b\in \R^{m}$ has been studied extensively.
For algorithms solving the linear sparse solution problem we refer to \cite{TroppWright}. Important references can also be found in \cite{BeckHallak2016}.

To the best of our knowledge there are no numerical algorithms specifically developed to find sparse solutions of \eqref{eq:f=0} except the ones described in \cite{Ku} which we will refer to as the $\ell_1$-method. We will later compare this method with our approach and therefore we describe the method in more detail.
Let $\| x \|_p, 0 < p < \infty$ be given as
\begin{equation}
\label{eq:p-norm}
  \| x \|_p = (\sum_i |x_i|^p)^{1/p}.
	\end{equation}
For $p\geq 1$ \eqref{eq:p-norm} defines the $\ell_p$-norm while for $0 < p < 1$ it is only a quasi-norm. In the sequel, we use $\|\cdot\|$ instead of $\|\cdot\|_2.$

The algorithms in \cite{Ku} are based on solving
\begin{equation}
\label{eq:problem:2}
\begin{array}{l}
    \min_x \| x \|_p\\
		\text{s.t. } f(x)=0
\end{array}
\end{equation}
for $0< p\leq 1$ and $f$ given as above, which is motivated by the fact that $\| x\|_p ^p\rightarrow \| x \|_0, p \rightarrow 0^+$ on a bounded set.
In particular, the $\ell_1$-norm algorithm described in \cite{Ku} is realized in the following way. Starting with $x_1=0$ one obtains a new approximation as $x_{k+1}=x_k+p_k,$ $k=1, 2,3,...,$ where $p_k$ is the solution to
\begin{equation}
\label{l1_lin}
\begin{array}{l}
    \min_p \|  p \|_1\\
		\text{s.t. } f_k + J_k p=0.
\end{array}
\end{equation}
Here we denote $f_k=f(x_k)$ and $J_k=\left( \partial f_i (x_k) / \partial x_j \right)_{ij},$  $i=1,\ldots,m,$ $j=1,\ldots,N$ is the Jacobian of $f(x)$ at $x=x_k$.
The problem \eqref{l1_lin} can be recast as a linear programming problem
\begin{equation}
\label{LP}
\begin{array}{l}
    \min_w c^T w\\
		\text{s.t. } Aw = b,\, \, w\geq 0
\end{array}
\end{equation}
where
\[
  c= {\bf 1}_{2N\times 1},\, A = \left( J_k, -J_k\right), \, b = -f_k, \, w = (u;v), \, p = u-v.
\]
In  \cite{Ku} it was shown that the method converges locally  to a solution (which is not necessarily sparse) with quadratic convergence rate. However,  global convergence was not proven.

There are other methods not directly applied to (\ref{eq:problem:1}) but that contains some ideas and properties  related to our approach and thus relevant to mention here.
In a series of papers \cite{BeckEldar2013,BeckEldar2013Proc,ShechtmanBeckEldar2014,BeckHallak2016} a general theory is developed for the problem
\begin{equation}
\label{eq:Beck}
\begin{array}{l}
    \min_x F(x)\\
		\text{s.t. } x\in C_s \cap B,
\end{array}
\end{equation}
where $F: \R^N\rightarrow \R$, $B$ is a closed and convex set, and $C_s = \left\{x\in \R^n : \| x \|_0 \leq s\right\}$. The theory is used for a number of applications and several algorithms are developed and analyzed. In our context, we note that in \cite{ShechtmanBeckEldar2014} an algorithm, GESPAR (greedy sparse phase retreival), is developed to solve a nonlinear overdetermined least squares problem based on a coordinate search  where the sparse (small) overdetermined nonlinear least squares subproblems are solved using a Gauss-Newton approach with line search. Convergence results for the gradient are derived.

In \cite{BahmaniBoufounosRaj2011} the problem (\ref{eq:Beck}) with $B=\R^n$ is considered with a coordinate search algorithm based on a local gradient search in a sparse solution set (gradient support pursuit). Estimates of the error in the iterates are developed using the size of the elements in the gradient at the sparse solution.

Furthermore, there are combinatorial methods that solve the nonlinear problem \eqref{eq:f=0} using cardinality constrains, see \cite{SunZhengLi2013}, which we do not consider here.

Here we present an alternative method, that we call a Greedy Gauss-Newton algorithm, that combines a greedy approach with the Gauss-Newton method \cite{Bj96}. The method is based on a line search where at  the $k$'th iterate we set
\begin{equation}
\label{eq:linesearch}
   x_{k+1} = x_k + \alpha_k p_k, \quad k=1,2, \ldots,
\end{equation}
where $p_k$ is the search direction and $\alpha_k$ is the step length. We start the iterations with $x_1=0$ or $x_1$ sparse enough. In every iteration we use the matrix $L_k$ consisting of the  columns of $J_k$ corresponding to the nonzero part of  $x_k$ and an additional  column of $J_k$, $J_k(:,t)$, to calculate the search direction  as $p_t = \argmin_p \| f_k - \left( L_k,J(:,t)\right)p \|$. The choice of $t$ is discussed and we analyze the two choices in detail. The first one is based on maximizing the descent of  $\| f(x) \|_2^2$ at $x=x_k$ in the direction $p_t$ and we call this method Maximum Descent (MD). The second idea of choosing $t$ is similar to orthogonal matching on the linear problem $\min \| f_k + J_k p \|_2$, see \cite{TroppWright}, and   consists of maximizing the angle between  $r_k = f_k + L_k L_k^+ f_k$ and $J_k(:,t)$, where $L_k^+$ is the pseudo inverse of $L_k$ \cite{IsraelGreville2003}. We denote our method based on orthogonal matching as OM.

The paper is organized as follows. In Section \ref{Sec:alg} we describe how to calculate $p_t$ and we show that it is a descent direction together with some useful corollaries. The MD algorithm is presented in Section \ref{Sec:MD} and OM in  Section \ref{Sec:OM}. In Section  \ref{Sec:con} we show results on global and local convergence together with the algorithm in pseudocode, and finally we give some numerical tests in Section \ref{Sec:num}.

\section{The algorithm}
\label{Sec:alg}

Here we describe the line search method \eqref{eq:linesearch} to find a sparse solution to \eqref{eq:f=0}. We start with $x_1=0$ or some sufficiently sparse vector.

At iteration $k,$  let $x_k$ contain $n_k$ nonzero elements at positions $i \in \Omega_k$ and zero elements at $i \in \overline{\Omega}_k$ where
\[
   \Omega_k = \left\{ i_1, i_2, \ldots , i_{n_k} \right\}, \, \, \overline{\Omega}_k = \left\{ 1, 2, \ldots , N \right\}\backslash \Omega_k.
\]

We use Matlab \cite{GolubVanLoan2013} inspired notation, that is,
$$x(i)=x_i,\quad x(:)=x, \quad x(1:n_k)=(x_1,...,x_{n_k})^T, \, \mbox{ and } x(\Omega_k) = \left(x_{i_1},\ldots,x_{i_{n_k}}\right)^T.$$

We aim at finding $p_k$  in \eqref{eq:linesearch} such that (i) $p_k$ is a descent direction and (ii) the update $x_{k+1}=x_k+ \alpha_k p_k$ is $(n_k+1)$-sparse for any $\alpha_k \in \R.$
The most straightforward approach would be to solve the linearized problem to \eqref{eq:f=0}, that is,
\begin{equation}
\label{eq:linprob}
    f_k + J_k \,p_k = 0.
\end{equation}
However, solving  \eqref{eq:linprob} for a sparse $p_k$ is not efficient enough for large $N$, \cite{TroppWright}.
Thus, for every $t\in \overline{\Omega}_k$ we define a projection $\Pi^t_k$ as
\begin{equation*}
\Pi_k^t(i,j)=\left\{
\begin{array}{ll}
1, & i=j \mbox{ and } i \in \Omega_k \cup \{ t \}, \\
\\
0,& \mbox{otherwise},
\end{array}\right.
\end{equation*}
where $i,j=1,\ldots,N.$  Then instead of \eqref{eq:linprob} we solve the minimization problem
\begin{equation}
\label{eq:linprob:t}
\begin{array}{l}
    \min \limits_p \frac{1}{2}\| p \|^2 \\
		\text{s.t. }  \min \limits_p \frac{1}{2}\|f_k + J_k \Pi^{t}_k p\|^2,
\end{array}
\end{equation}
with $t\in \overline{\Omega}_k$  to obtain $p_k.$ We choose $t\in \overline{\Omega}_k$  by two different methods: MD, $t=t_{MD}$, or  OM, $t=t_{OM}$, that we describe in details in the coming subsections.

 Let $p_t$ be a solution of \eqref{eq:linprob:t} for  $t \in \overline{\Omega}_k.$ It is clear that $p_t$ for any $t \in \overline{\Omega}_k$ satisfies the sparsity requirement (ii). Indeed,  $p_{t}(\overline{\Omega}_k \setminus \{t\})=0$ for any $t\in \overline{\Omega}_k.$  Below we discuss when $p_t$ is a descent direction.

 Denote $L_k=J_k(:,\Omega_k),$ then the remaining non-zero part of $p_t$, that is, $q_t =(p_{t}({\Omega}_k)^T$, $p_t(t))^T \in \R^{n_k+1}$ is the solution to
   \begin{equation}
    \label{eq:linprob:t:q}
    \begin{array}{l}
        \min \limits_{q} \frac{1}{2}\| q\|^2 \\
        \text{s.t. } \min\limits_q \frac{1}{2}\| f_k + \big(L_k,  J(:,t)\big) q\|^2,
    \end{array}
    \end{equation}
that is,

 \begin{equation}
 \label{eq:qt}
 q_t=-\big(L_k, J_k(:,t)\big)^{+}f_k.
  \end{equation}
Note that $q_t$ is the unique minimum of $\| f_k + \big(L_k,  J(:,t)\big) q\|$  if  $\rank \left(\big(L_k, J_k(:,t)\big) \right) \geq n_k+1.$


\begin{lemma} \label{lemma:descent}
Let  $p_t$ be a solution of \eqref{eq:linprob:t}, $J_k$ and $f_k$ be given as above. Then
\begin{equation} \label{eq:descent}
-p_t^TJ_k^Tf_k =f_k^T L_k L_k^+ f_k+ \dfrac{|f_k^T(I-L_k L_k^+)J_k(:,t)|^2}{\|(I-L_k L_k^+)J_k(:,t)\|^2}
\end{equation}
and $p_t$ is a descent direction of $1/2\|f(x)\|^2$ at $x=x_k$ if and only if
\begin{equation} \label{ineq:descent}
f_k^T L_k L_k^+ f_k+ \dfrac{|f_k^T(I-L_k L_k^+)J_k(:,t)|^2}{\|(I-L_k L_k^+)J_k(:,t)\|^2}>0.
\end{equation}
\end{lemma}
\begin{proof}
Let $t\in \overline{\Omega}_k$, $a=J_k(:,t)$, and $P=I-L_k L_k^+.$  Observe that $L_kL_k^+$ and $P$ define the orthogonal projections on $\cR (L_k)$ and $\cR (L_k)^{\bot},$ respectively.

We have
\[
-p_t^TJ_k^Tf_k=-q_t^T \big(L_k,a\big)^+f_k.
\]

Theorem 2 in Ch.7 Section 5 in \cite{IsraelGreville2003} yields
\begin{equation}\label{eq:pinv:[Lk,a]}
\big(L_k,a\big)^+=
\begin{pmatrix} \
L_k^+ -L_k^+ab^T\\
b^T
\end{pmatrix}
\end{equation}
where
\[
b^T=\left\{
\begin{array}{cl}
(Pa)^+, &Pa\not=0,\\
\\
\dfrac{a^T (L_k^+)^T L_k^+}{1+a^T (L_k^+)^T L_k^+ a},& Pa\not=0.
\end{array}
\right.
\]

Thus,  using \eqref{eq:qt}  and \eqref{eq:pinv:[Lk,a]} we obtain
\begin{equation}
-p_t^TJ_k^Tf_k=-q_t^T \big(L_k,a\big)^+f_k=f_k^T L_k L_k^+ f_k+ \dfrac{|f_k^TPa|^2}{\|Pa\|^2}.
\end{equation}
The second claim of the lemma follows from \eqref{eq:descent} and the definition of a descent direction.
\end{proof}

\begin{corollary}
\label{cor:desc}
The solution $p_t$ of \eqref{eq:linprob:t} is a descent direction of $1/2\|f(x)\|^2$ at $x=x_k$ unless
$f_k \in \cR (L_k)^{\perp}$ and $J_k(:,t)\in \cR (L_k)$ simultaneously.
\end{corollary}
\begin{proof}
Observe that $L_kL_k^+$ is positive semi-definite. Thus the two last terms  in \eqref{ineq:descent} are non-negative. Assume that $f_k^T L_k L_k^+ f_k=0.$ Then $f_k \in \cR(L_k)^{\perp}$ and the last term in \eqref{eq:descent} is equal to zero only if $a \in \cR(L_k).$
\end{proof}

It is  clear from \eqref{eq:descent} that adding an extra column of $J_k$ will improve the descent as long as the added column does not belong to $\cR(L_k)$. We formulate it as a corollary.
\begin{corollary}
The descent of $p_t,$ $t\in \overline{\Omega}_k$ is not less than the descent of $p_0$ where $p_0(\overline{\Omega}_k)=0$ and $p_0(\Omega_k)$ is the solution to
\begin{equation}
    \begin{array}{l}
        \min \limits_{d} \frac{1}{2}\left\| x(\Omega_k)+ d\right\|^2 \\
        \\
    	\text{s.t. } \min\limits_d \frac{1}{2}\| f_k + L_k d\|^2.
    \end{array}
    \end{equation}
\end{corollary}

\begin{proof}
Simple calculations show that $-p_0^TJ_k^Tf_k=f_k^T L_k L_k^+ f_k.$ Together with
\eqref{eq:descent} it implies  $p_0^TJ_k^T-p_t^TJ_k^T\geq 0.$
\end{proof}

After $p_k$ is constructed, the step length $\alpha_k$ is found by using a standard step length algorithm, see \cite{Kelley1999}, that satisfy the Goldman-Anmijo rule. We get $x_{k+1} = x_k + \alpha_k p_k$ that has at least $n_k + 1$ nonzero elements
and $\Omega_{k+1}=\Omega_k \cup \{t_*\}.$ If the step length $\alpha_k$ is too small it  indicates that the descent is insufficient and we restart the algorithm with a sparse enough $x_1$ where the positions and the values of nonzero elements are chosen randomly, see Section \ref{Sec:Pseudo}.

If there are elements in $x_{k+1}$ close to zero it could make sense to put these values to zero and then recalculate the set of non-zero entries  $\Omega_{k+1}$. This approach would be however very much problem dependent and we do not consider it here.
\subsection{Maximum descent method (MD)}
\label{Sec:MD}

MD is based on choosing $p_k=p_{t_{MD}}$ where
\[
t_{MD}=\argmax\limits_{t\in \overline{\Omega}_k} \left(-p_t^T J_k^T f_k\right)
\]
or, equivalently,
\begin{equation}
\label{eq:tMDqt}
t_{MD}=\argmax\limits_{t\in \overline{\Omega}_k} \left(-q_t^T \big(L_k,J(:,t)\big)^T f_k\right).
\end{equation}

The next lemma gives us the explicit formula for computing $t=t_{MD}.$
\begin{lemma} \label{lemma:MD}
Let $p_t$ be the solution to \eqref{eq:linprob:t} for $t\in \overline{\Omega}_k.$
If there exists a $t\in \overline{\Omega}_k$ such that $p_t$  is a descent direction of $1/2 \|f(x)\|^2$ at $x_k$, then the maximum descent direction is given as $p_k=p_{t_{MD}}$ where
\begin{equation} \label{eq:tMD}
t_{MD}=\argmax \limits_t \frac{\left| f_k^T (I-L_kL_k^+) J_k(:,t)\right|}{\|(I-L_kL_k^+)J_k(:,t)\|}.
\end{equation}

Moreover, $p_{t_{MD}}$ provides the minimum of the norm $\|f_k+J_kp_t\|,$ i.e.,
\[
p_{t_{MD}}=\argmin\limits_{p_t}\|f_k+J_kp_t\|.
\]
\end{lemma}
\begin{proof}
Let $t\in \overline{\Omega}_k$, $a=J_k(:,t)$,  $P=I-L_k L_k^+$, and $S=Pa a^T P/\|Pa\|^2$, where $P$ and $S$ define the orthogonal projections on $\cR (L_k)^{\bot}$ on $\cR(Pa),$ respectively.
Descent is given by \eqref{eq:descent} where
the first term in the right hand side, $f_k^TL_kL_k^+f_k$, does not depend on $a$ and thus the maximum descent  is achieved when $|f_k^TPa|/\|Pa\|$ is maximum. Thus, we obtain the expression in \eqref{eq:tMD}.

To prove the second claim of the theorem we compute the squared norm using
the expression for $q_t$ in \eqref{eq:qt}
\begin{equation} \label{eq:norm:implicit}
\|f_k+J_k \Pi_k^t p_t\|^2=\|f_k+ \big(L_k,a\big) q_t\|^2=
\left\|\left(I-\big(L_k,a\big)\big(L_k,a\big)^+\right)f_k\right\|^2.
\end{equation}

Using \eqref{eq:pinv:[Lk,a]} we obtain
\begin{equation}
\begin{split}
&\left\|\left(I-\big(L_k,a\big)\big(L_k,a\big)^+\right)f_k\right\|^2=\|(P-Pab^T)f_k\|^2=\|(P-S)f_k\|^2\\
&=fk^T(P^2-PS)f_k=f_k^TPf_k-\frac{|f_k^TPa|^2}{\|Pa\|^2} \geq 0
\end{split}
\end{equation}
The term $f_k^TPf_k$ does not depend on $a$ and the norm $\|f_k+J_k \Pi_k^t p_t\|$ reaches its minimum when $|f_k^TPa|/\|Pa\|$ is maximum.
\end{proof}

From Corollary \ref{cor:desc} and Lemma \ref{lemma:MD} it is clear that $p_k =p_{t_{MD}}$ is always a descent direction if $\rank(J_k)>\rank(L_k).$

Let us assume that $q_t$ in (\ref{eq:qt}) is calculated with a QR-decomposition, see \cite{GolubVanLoan2013}, $N\gg  m \gg n_k$,  and  $t_{MD}$ is calculated using (\ref{eq:tMDqt}). Then  the complexity (number of flops, i.e., one addition, subtraction,
multiplication, or division of two floating-point numbers) of MD in  iteration $k$ is $2mn_k^2 +(m+1)(n_k+1)(N-n_k)$. If instead we use \eqref{eq:tMD}, the complexity is $2m(n_k+1)(N-n_k)$. Assuming that the term including $N-n_k$ is the largest the complexity of MD can be reduced by accepting a descent large enough without considering the whole set $\overline{\Omega}_k$. However, we have not considered this generalization here.

\subsection{Orthogonal matching method (OM)}
\label{Sec:OM}

Let $L_k = J_k(:,\Omega_k)$ as before and consider

\begin{equation}
\label{eq:linmin:explicit}
\begin{array}{l}
    \min\limits_d \frac{1}{2}\| d\|^2 \\
		\text{s.t. } \min\limits_d \frac{1}{2}\| f_k + L_k d\|^2.
\end{array}
\end{equation}

 The solution of \eqref{eq:linmin:explicit} is $d_k=-L_k^{+}f_k$ which is the unique minimum to $\| f_k + L_k d\|$ if $\rank(L_k) \geq n_k$, and the minimum norm solution otherwise.

OM aims at finding the column $J_k(:,t_{OM})$ that is the most strongly correlated with the linear residual $r_k=f_k+L_kd_k$, i.e.,
\[
t_{OM}=\argmax\limits_{t\in \overline{\Omega}_k} \left|r_k^T \dfrac{J_k(:,t)}{\|J_k(:,t)\|}\right|
\]
or equivalently,
\begin{equation}
\label{eq:tOM}
t_{OM}=\argmax\limits_{t\in \overline{\Omega}_k} \dfrac{\left|f_k^T (I-L_kL_k^+)J_k(:,t)\right|}{\|J_k(:,t)\|},
\end{equation}
to obtain $p_k=p_{t_{OM}}.$

Following the assumptions made for MD regarding complexity analysis we get the complexity of OM to be $4mn_k^2 + 2m(N-n_k)$ where the first term is the calculation of $r_k$ and $q_{t_{OM}}$ in (\ref{eq:qt}) and the second is from solving the maximization problem in (\ref{eq:tOM}).


Let us consider \eqref{eq:linprob:t} where we set $p(\Omega_k)=d_k,$ that is,
\begin{equation}
\label{eq:linprob:t:newp}
\begin{array}{l}
    \min \limits_p \frac{1}{2}\| p \|^2\\
    \text{ s.t. }
            \left\{
            \begin{array}{l}
            \min \limits_p \frac{1}{2}\|f_k + J_k \Pi^{t}_k p\|^2\\
            p(\Omega_k)=d_k
            \end{array}
            \right.
\end{array}
\end{equation}

 Then \eqref{eq:linprob:t:q} can be
rewritten as
\begin{equation*}
\begin{array}{l}
    \min\limits_\delta \frac{1}{2} \left\|
        \begin{pmatrix}
    d_k\\
    \delta
    \end{pmatrix}
    \right\|^2\\
    \\
		\text{s.t. } \min\limits_\delta \frac{1}{2}\| f_k + L_k d_k+J(:,t)\delta\|^2,
\end{array}
\end{equation*}

or, equivalently,
\begin{equation}
\label{eq:delta}
\begin{array}{l}
    \min\limits_\delta |\delta| \\
		\text{s.t. } \min\limits_\delta \frac{1}{2}\| (I-L_k L_k^+)f_k + J(:,t)\delta\|_2
\end{array}
\end{equation}
with the solution
\begin{equation}
\label{eq:delta:t}
\delta_t=-\frac{J_k(:,t)^T}{\|J_k(:,t)\|^2} (I-L_k L_k^{+})f_k.
\end{equation}

Hence, the solution to \eqref{eq:linprob:t:newp} is $\tilde{p}_t$ where $\tilde{p}_t(\Omega_k)=d_k,$ $\tilde{p}_t(t)=\delta_t$ and $\tilde{p}_t(\overline{\Omega}_k\setminus\{t\})=0.$
%

\begin{lemma}\label{lemma:OM}
Let $p_t$ be a solution to \eqref{eq:linprob:t} and $\tilde{p}_t$ to \eqref{eq:linprob:t:newp} for $t\in \overline{\Omega}_k,$ and $t_{OM}$ be given by \eqref{eq:tOM}. If there exists a descent direction among $p_t$ then  $p_{t_{OM}}$ and $\tilde{p}_{t_{OM}}$ are  descent directions. Moreover, $\tilde{p}_{t_{OM}}$ gives the minimum norm of $\|f_k+J_k\tilde{p}_t\|,$ i.e.,
\[
\tilde{p}_{t_{OM}}=\argmin \limits_{\tilde{p}_t} \|f_k + J_k\tilde{p}_t\|.
\]
\end{lemma}
\begin{proof}
Let $P$ define the orthogonal projections on $\cR (L_k)^{\bot}$, i.e.,$P =I-L_k L_k^+.$
From Corollary \ref{cor:desc},it i seen that $p_{t_{OM}}$ is a descent direction. Indeed, if $f_k \not\in \cR (L_k)^{\perp}$ then any $p_{t}$ gives a descent. Assume that $f_k \in \cR (L_k)^{\perp}.$ Let $p_{t^*},$ $t^* \in \overline{\Omega}_k$ be a descent direction.  Hence,  $J_k(:,t^*) \not\in \cR (L_k),$ that is, $|f_k^TPJ_k(:,t^*)|>0$ which implies $|f_k^T P J(:,t_{OM})| >0.$

Let $t\in \overline{\Omega}_k,$ $a=J_k(:,t)$ and $Q=a a^T/\|a\|^2$ define the orthogonal projections on $R(a).$
To show that $\tilde{p}_{t_{OM}}$ gives a descent we calculate
\[
-\tilde{p}_t^T J_k^T f_k=-(d_k^T,\delta^T)
\begin{pmatrix}
L_k^T f_k\\
a^T f_k
\end{pmatrix}.
\]

Using the formulas for $d_k$ and \eqref{eq:delta:t} we have
\[
-\tilde{p}_t^T J_k^T f_k=f_k^TL_kL_k^+f_k+\dfrac{f_k^TPaa^Tf_k}{\|a\|^2} \geq 0
\]
as $Paa^T$ is positive semi-definite which can be seen by looking at the eigenvalue equation $Paa^Tu = \lambda u$ giving  $\lambda \geq 0$.
Similarly to as above,  $f_k \not\in \cR (L_k)^{\perp}$ implies that $\tilde{p}_t$ is a descent direction for any $t\in \overline{\Omega}_k.$ Assume that this is not the case and $f_k \in \cR (L_k)^{\perp}.$ Then $J(:,t_{OM}) \not \in \cR (L_k)$
which implies $f_k^TPaa^Tf_k/\|a\|^2>0$ for $a=J(:,t_{OM})$ and $\tilde{p}_{t_{OM}}$ gives a descent direction.

 To show that  $\tilde{p}_{t_{OM}}$ provide the minimum norm we compute
\begin{equation*}
\begin{split}
\|f_k+J_k\tilde{p}_t\|^2&=\|Pf_k-Q Pf_k\|^2=f_k^TP(I-Q)Pf_k\\
&=f_k^TPf_k-f_k^TPQPf_k =f_k^TPf_k-\dfrac{|f_k^TPa|^2}{\|a\|^2}\geq 0.
\end{split}
\end{equation*}
The term $f_k^TPf_k$ does not depend on $a$ and the norm $\|f_k+J_k\tilde{p}_t\|$ reaches its minimum when $|f_k^TPa|/\|a\|$ is maximum.
\end{proof}

From Lemma \ref{lemma:OM} it follows that one can use $p_k=\tilde{p}_{t_{OM}}$ instead of $p_k=p_{t_{OM}}.$
However the complexity of this approach would be only $2mn_k$ less  than $OM$. As Lemma \ref{lemma:OM} and Lemma \ref{lemma:MD} imply
\[
\|f_k+J_kp_{t_{MD}}\|\leq
\|f_k+J_k p_{t_{OM}}\|\leq \|f_k+J_k\tilde{p}_{t_{OM}}\|
\]
and we have not seen any real advantages of this approach compared to OM, we do not consider it further.

\subsection{Comparison and generalizations of OM and MD}

There are some interesting common features between MD and OM. In  \eqref{eq:tMD} we notice that the new column is chosen as to maximize the angle between the vectors $f_k$ and $v^t_k =(I-L_kL_k^+) J_k(:,t)/\|(I-L_kL_k^+)J_k(:,t)\|$. Geometrically this means that we choose the column $J(:,t)$ whose projection onto $\cR (L_k)^{\perp}$ is as parallel as possible to the nonlinear residual $f_k$. In OM we instead choose $t_{OM}$ from (\ref{eq:tOM}) which is the maximization of the angle between the linear residual $r_k$ and $J_k(:,t)$. This is the same Orthogonal Mathing principle as for linear problem \cite{TroppWright} but here on the linearized problem $\min_p \|f_k + J_k p\|$.

From a complexity point of view the two methods are comparable if we assume that $N\gg m \gg n_k$ but if $n_k \approx m$ MD will be more expensive since the large term is $\mathcal{O} (m^2 (N-n_k))$ compared to $\mathcal{O} (m(N-n_k))$ using OM.

We note  that when $n_k = \text{rank}(J_k)$ no column will be added and we then choose to remain in the corresponding subspace.

There are some more or less  obvious  variants or generalizations of MD and OM and we mention some here. Firstly, more than one column can be added in every iteration simplifying the algorithm and possibly making it more efficient. Secondly, the search of the columns may not be exhaustive, i.e., as soon as a column is found satisfying the criteria for being added the search can be terminated. Specifically, this is an attractive approach for MD since only sufficient descent is necessary not necessarily maximum descent. Finally, it is  possible to iterate in the corresponding subspace at each step possibly using a line search or any other approach.

\section{Convergence properties}
\label{Sec:con}

The global convergence is given by the following classical theorem that we state here for the sake of completeness. For the reference see Theorem 6.3.3. in \cite{book:DennisSchnabel} or Theorem 14.2.14 in \cite{book:OrtegaRheinboldt}.

\begin{theorem}[Global Convergence of a Descent method]
\label{th:global:main}
Let $F:D \subset \R^N \to \R^1$ be continuously differentiable on the open convex set $D$ and assume that $\nabla F$ satisfy the Lipschitz condition
\[
  \|\nabla F(x) - \nabla F(x) \|_2 \leq \gamma \|x-z\|
\]
for every $x,z \in D$ and some $\gamma>0.$  Given $x_1\in D$ assume that the level set
$\Lambda=\{x \in D\,| \,F(x) \leq F(x_1)\}$ is compact. Consider the sequence $\{x_k\}$  defined by \eqref{eq:linesearch} with $\alpha_k \geq 0$ satisfying the Armijo-Goldstein condition,
and $-p_k^T \nabla F(x_k) > 0$ for all $k \in \N.$ Then $\{x_k\} \in \Lambda$ and
\begin{equation}\label{eq:lim:global}
\lim\limits_{k\to\infty} \dfrac{p_k^T \nabla F(x_k)}{\|p_k\|}=0.
\end{equation}
\end{theorem}

Next we show that the algorithm in Section \ref{Sec:Pseudo} with $p_k$ chosen using  MD method or OM has the same convergence properties as the Gauss-Newton method for underdetermined nonlinear problems.

\begin{lemma} \label{lemma:AvsGN}
Let $f$ be given as in \eqref{eq:f=0}, $x_1 \in D$ where $D\subset \R^N$ is a convex open set such that  $\Lambda=\{x\in D \,|\, \|f(x)\|\leq \|f(x_1)\|\}$ is compact. Consider the sequence $\{x_k\}$ given by \eqref{eq:linesearch} with the descent direction $p_k$ chosen using MD  or OM, and $\alpha_k>0$ satisfying the Armijo-Goldstein rule.
If $\rank(J(x))=\rho\leq m$ for all $x \in \Lambda$ then there is $k_\rho\in \N$ such that for $k\geq k_\rho$
\begin{equation} \label{eq:help:2}
-p_k^TJ_k^T f_k=f_k^TJ_k J_k^+ f_k.
\end{equation}
\end{lemma}
\begin{proof}
Under the conditions of Theorem \ref{th:global:main} $x_k\in \Lambda,$ see 14.2.3 in \cite{book:OrtegaRheinboldt}, and thus $\rank(J_k)=\rho,$ $k \in \N.$ Let $a=J_k(:,t_*)$ where where $t_*=t_{MD}$ or $t_*=t_{OM},$ see \eqref{eq:tMD} and \eqref{eq:tOM}. From Lemma \ref{lemma:descent} and Corollary \ref{cor:desc}, $\rank(L_k)=\rho$  for all $k\geq k_\rho$ for some $k_\rho  \in \N$ and thus, $(I-L_kL_k^+)a=0.$ Hence,  from \eqref{eq:descent} we have
\begin{equation}
\label{eq:help:1}
-p_k^TJ_k^Tf_k=f_k^T L_k L_k^+ f_k.
\end{equation}

Without loss of generality assume $J_k=\Big(L_k,\overline{L}_k\Big)$ and let $E\in \R^{N\times N}$ be a product of elementary matrices such that
\[
J_k=\Big(L_k,\overline{L}_k\Big)=\Big(L_k,0\Big)E.
\]

Then
\[
J_k J_k^+=\Big(L_k,0\Big)E E^{-1}\Big(L_k,0\Big)^+=
\Big(L_k,0\Big)
\begin{pmatrix}
L_k^+\\
0
\end{pmatrix}
=L_kL_k^+
\]
which yields \eqref{eq:help:2}.
\end{proof}

Notice that from Lemma \ref{lemma:AvsGN} the algorithm becomes equivalent to the Gauss-Newton method only starting from some $k_\rho$th iterate, when we already has (hopefully) reached the vicinity of a sparse local minimum of $1/2 \|f\|^2,$ say $x^*.$ This minimum is a solution to $f(x)=0$ if $\rank(J(x^*))=m$ but this is not necessarily the case when $\rank(J(x^*))<m.$
In practice we exclude the convergence to a stationary point $x^*$ giving $\|f(x^*)\|>0$ by restarting the algorithm. We also do a restart when  $p_k$ fails to give a significant descent, see Section \ref{Sec:Pseudo}.

Let $\{x_k\}$ be generated by the Greedy Gauss-Newton method and $\{x_k\} \to x^*$ where $f(x^*)=0.$ Then the convergence rate is quadratic given $\alpha_k=1$ in a vicinity of $x^*$, see \cite{book:DennisSchnabel}. However, from Lemma \ref{lemma:AvsGN} this rate of convergence is only guarantied for $k>k_\rho$. With next proposition we show that this assumption on $k$ can be omitted.

\begin{proposition}[Rate of Convergence]
Let $f$ be given as in \eqref{eq:f=0} and $\hat{x}\in \R^N$ be such that $f(\hat{x})=0.$  Let the sequence $\{x_k\}$ given by \eqref{eq:linesearch} with the descent direction $p_k$ chosen using  MD or OM and $\alpha_k=1$ converges to $\hat{x}$ as $k\to \infty.$ If $\|p_k\|\leq C \|f_k\|$ for all $k\geq K,$ for some $K\in \N,$ then $\{x_k\}$ converges to $\hat{x}$ quadratically.
\end{proposition}
\begin{proof}
Let  $A_k(:,\Omega_k \cup \{t_*\})=J_k(:,\Omega_k \cup \{t_*\})$ and $A_k(:,\overline{\Omega}_k \setminus\{t_*\})=O$ where $t_*=t_{MD}$ or $t_*=t_{OM}.$ Then $p_k=-A_k^+ f_k$ and $\|A_k^+\|\leq C.$ In a vicinity of $\hat{x}$ the Taylor expansion is valid
 \[
 f(\hat{x})=f(x)+J_k(\hat{x}-x_k)+r(x_k)=f_k+A_k(\hat{x}-x_k)+r_k
 \]
 with $r_k=O(\|x-\hat{x}\|^2)$ as the Hessian is continuous and thus uniformly bounded in a closed neighbourhood of $\hat{x}.$

 We have
  \[
 A_k^+f(\hat{x})=A_k^{+}f_k+A_k^{+}A_k(\hat{x}-x_k)+A_k^{+}r(x).
 \]
 Remembering that $f(\hat{x})=0$ and $A_k^{+}A_k=I$ we obtain
 \[
 x_k-\hat{x}=A_k^{+}f_k+A_k^{+}r(x).
 \]

 Next,
 \[
 x_{k+1}-\hat{x}=(x_k-\hat{x})-A_k^{+}f_k=A_k^{+}r(x)=O(\|x_k-\hat{x}\|^2)
 \]
which completes our proof.
\end{proof}
%

\subsection{The Greedy Gauss-Newton Algorithm in pseudocode}
\label{Sec:Pseudo}

Below we outline the algorithm we use in our numerical tests.
For the values of the constants in step 1. we refer to the numerical tests in Section \ref{Sec:num}. The parameter
$k_{max}$ stands for the maximum number of iterations (counting throughout restarts), $\varepsilon_f$, $\delta_x$, $\delta_{\alpha}$, $tol$, and $\Delta_{grad}$ are tolerances.

In step 14. the sign $"\circ"$ stands for the Hadamard product and $\rm{rand}(N,1)$ returns a vector of $N$ uniformly distributed random numbers in the interval $(0,1),$ and $prob \in (0,1].$

The merit function $\phi(\alpha)$ in step 10. is given as $ \phi(\alpha)= \| f(x_k + \alpha p_k)\|_2^2/2$.
\\ \\

\noindent
{\bf Greedy Gauss-Newton Algorithm}\\[2mm]
Predefined functions are $f : \R^{N}\rightarrow \R^{m}$ and Jacobian $J(x):\R^{N}\rightarrow \R^{m\times N} $, $m<N$ \\[2mm]
1.  \indent  \indent \indent {\bf Input}: $k_{max}$, $\varepsilon_f$, $\delta_x$, $\delta_{\alpha}$, $tol$, $\Delta_{grad}$, $prob$\\
2. ‎ \indent  \indent \indent $k=1$, $x_1 = 0$, $\Omega_1 = \emptyset$, $n_{restarts}=0$\\
3.  \indent  \indent \indent {\bf while} $\| f(x_k) \| > \varepsilon_f$ and  $k<k_{max}$\\
4.  \indent  \indent \indent  \indent Find $t_{max}$ from (\ref{eq:tMD}) if MD or (\ref{eq:tOM}) if OM (or any other method)\\	
5. \indent  \indent \indent  \indent {\bf if} the maximum in (\ref{eq:tMD}) or (\ref{eq:tOM}) respectively is larger  than $tol$\\
6. 	 \indent  \indent 	\indent \indent  \indent Set $\Omega_{k+1} = \Omega_k \cup t_{max}$\\
	  \indent \indent  \indent 	\indent \indent	{\bf else}\\
7. \indent  \indent 	\indent \indent  \indent Set $\Omega_{k+1} = \Omega_k$\\
	 \indent \indent  \indent 	\indent \indent  {\bf end}\\
8.  \indent  \indent \indent  \indent Compute  $p_k = -J(:,\Omega_{k})^+f(x_k)$\\
9.	 \indent \indent \indent \indent Find $\alpha_k$ using the merit function $\phi (\alpha)$ \\
10.	 \indent \indent  \indent  \, Set  $x_{k+1} = x_k + \alpha_k p_k$\\
11.	\indent  \indent \indent \, {\bf if} $\alpha_k<\delta_{\alpha}$ or $\| J_k^T f_k \|/ \| f_k \|< \Delta_{grad}$\\
12.  	 \indent  \indent       \indent \indent \, $n_{restarts} =n_{restarts} +1$\\
13.		 \indent  \indent 		\indent \indent   \, Set $x_{k+1}  =(2\rm{rand}(N,1)-1)\circ(\rm{rand}(N,1)<prob)$\\
14.		 \indent  \indent 	\indent \indent  \,  Update $\Omega_{k+1} = \left\{ i:|x_k(i)|>\delta_x \right\}$ \\
  \indent \indent  \indent \indent  \indent  {\bf end}  \\
15.		 \indent  \indent \,  Update $k=k+1$\\
  \indent \indent  \indent 	 \indent {\bf end}\\
16.   \indent \indent  \indent  \,   Update $\Omega_{k+1} = \left\{ i:\,|x_k(i)|>\delta_x \right\}$ and  $x_{k+1}( \overline{\Omega}_{k+1})=0$\\
17.  \indent  \,   {\bf Output}: Solution to $f(x)=0$ or  if $k=k_{max}$ the vector $x_{k_{max}}$\\

A restart, see step 14., is performed if either the step length is too small indicating not enough descent, or if the gradient is small while the norm of $f$ is not small, see step 12. The first case appears when the Gauss-Newton method does not converge locally, i.e., the solution has a large residual $f$ and/or a small curvature, see \cite{ErWeGuSo2005} for details. The second case for a restart may occur when the algorithm is converging to a local minima where the norm of $f$ is not close to zero.

In the next section we use  $k_{max}=200,$  $\delta_x=10^{-8},$ $\varepsilon_f=10^{-13},$ $\delta_\alpha=10^{-3}$, $tol=10^{-10},$ and $\Delta_{grad} = 10^{-16}$. The other constants vary for different problems and are given below.

%
%
%
%

\section{Numerical tests}
\label{Sec:num}

We test our method on three different problems where the solution space is known. The first is a small problem that is considered in \cite{Ku}. The second and the third one have quadratic and exponential nonlinearities, respectively. These are large problems which size can be changed.
We illustrate the results from both qualitative and quantitative point of view and test the algorithm versus $\ell_1$-method described in \cite{Ku}.

\subsection{Small test problem} \label{Sec:num:small}

Let $f$ in \eqref{eq:f=0} be given as
\[
f(x)=Ax+\phi(x)-y
\]

where
\[
A=
\left(
 \begin{array}{rrrrrrrr}
-3.933& 0.107& 0.126& 0& -9.99& 0& -48.83& -7.64\\
    0& -0.987& 0& -22.95& 0& -28.37& 0& 0\\
    0.0002& 0& -0.235& 0& 5.67& 0& -0.921& -6.51\\
    0 &1 &0 &-1 &0 &-0.168 &0 &0\\
    0 &0& -1& 0& -0.196& 0& -0.0071& 0
    \end{array}
 \right),
\]
\[
\phi(x)=
 \left(
 \begin{array}{r}
 -0.727 x(2)x(3)+8.39x(3)x(4)-684.4x(4)x(5)+63.5x(4)x(7)\\
        0.949x(1)x(2)-1.578x(1)x(4)-1.132x(4)x(7)\\
        -0.716x(1)x(2)-1.578x(1)x(4)+1.132x(4)x(7)\\
        -x(1)x(5)\\
        x(1)x(4)
 \end{array}
 \right),
\]
\[
 y = (0.999, -1.4185, -0.5670, -0.0084, 0.0196)^T.
 \]
 We run the $\ell_1$-method and both  MD and OM starting with $x_1=0 \in \R^8.$
 It turns out that for this set up  MD and OM are equivalent.

 All the methods converged to the same sparse solution $\hat{x}=(0,0,0,0,-0.1,0.05,0,0)^T.$ After three iterations we obtained $\|f(x_3)\|<1e-15$.
 Below we print the matrix $X_{l_1}=(x_1,x_2,x_3)$ where $x_k,$ $k=1,2,3,$ are the iterates obtained using the $\ell_1$-method
  \[X_{l_1}=
  \left(
 \begin{array}{rrr}
            0  & -1.94e-15&  5.70e-16\\
            0  & 1.64e-14&   8.53e-15\\
            0  & 1.21e-15&  -2.03e-17\\
            0  & 1.91e-14&   9.99e-15\\
            0  &-0.1&  -0.1\\
            0  & 0.05&   0.05\\
            0  & 2.38e-15&   1.44e-14\\
            0  &-5.83e-15&  -2.42e-15
           \end{array}
 \right),
 \]
  and  $X=(x_1,x_2,x_3)$ with $x_k,$ $k=1,2,3,$ obtained using OM (or MD)
 \[
 X =
 \left(
 \begin{array}{rrr}

            0            &0            &0\\
            0            &0&            0\\
            0            &0&            0\\
            0            &0&            0\\
            0            &0&  -0.1\\
            0           &0.05&   0.05\\
            0            &0&            0\\
            0            &0&            0
           \end{array}
 \right).
 \]

The matrices above give a good illustration of the difference between the two algorithms. In particular, the choice of the parameter $\delta_x$ plays more significant role for the $\ell_1$ -method then for the Greedy Gauss-Newton algorithm. Moreover, the maximum sparsity of a solution obtained by the Greedy Gauss-Newton algorithm not grater than $m,$ which can not be guaranteed by the $\ell_1$-method.

\subsection{Quadratic test problem} \label{Sec:Quadratic}

Consider the quadratic function
\begin{equation}
\label{eq:qp}
    f(x) =
		A(x-\bar{x}) +\frac{1}{2}
		\left(
		\begin{array}{c}
		    (x-\bar{x})^T H_1(x-\bar{x})\\
				(x-\bar{x})^T H_2(x-\bar{x})\\
				\vdots\\
				(x-\bar{x})^T H_m(x-\bar{x})
    \end{array}
		\right),
\end{equation}
where $A, \, H_i\in \R^{N\times N}$, $i=1,...,m.$

Let $s, n$ be such that $1\leq s<n+s\leq N$ and
\[
   Q = \left( Q_1,  Q_2 \right), Q_1 \in \mathbb{R}^{(n+s)\times n}, Q_2 \in  \mathbb{R}^{(n+s)\times s}, Q^TQ = I.
\]
We define
\[
    A = \left( BQ_1^T,C \right), \, \mbox{ and } \,
    H_i =
		\left(
		\begin{array}{cc}
		    Q_1T_iQ_1^T & S_i\\
				S_i^T & R_i
    \end{array}
		\right)
\]
where 
$B,$ $C,$ $T_i,$ $S_i,$ and $R_i,$ $i=1,...,m,$ are all random matrices of the corresponding sizes whose elements are uniformly distributed in $(-1,1)$. We assume that $\bar{x}$ is $(n+s)$ - sparse with $(n+s)$ first non-zero elements.
Let $\bar{z}=\bar{x}(1:n+s)$ then any $x$ such that
\begin{equation} \label{eq:x:qp}
x-\bar{x}=
\begin{pmatrix}
z-\bar{z}\\
0
\end{pmatrix}
=
\begin{pmatrix}
Q_2y\\
0
\end{pmatrix}, \quad y \in \R^s,
\end{equation}
is a solution to \eqref{eq:qp}.
Moreover, as one can always find $y \in \R^s$ such that $z=Q_2y+\bar{z}$ has additional $s$ zeros, we conclude that there are solutions $x$ of sparsity $n.$

The Jacobian, $J(x)\in \R^{m\times N}$,  of $f$ is given by
\[
  J_{ij}(x) = a_{ij} + e_j^TH_i (x-\bar{x}), \, i=1, \ldots , m,\,  j=1, \ldots N, 
\]
where $e_j$ is the $j$'th unit vector, and $f''_i = H_i$.

Thus, for $x$ as in \eqref{eq:x:qp} we obtain
\[
J(x)=(BQ_1^T, C+J_{12}), \quad J_{12}=e_j^T S_i^TQ_2 (y-\bar{y})
\]
which most probably has  rank $m.$

All the tests we run with $N=100$, $m=20$, $s=2$, $prob=0.02$ and with the constants given in Section \ref{Sec:Pseudo}.

In Figures \ref{fig:TestQ1:MD} - \ref{fig:TestQ3}
we demonstrate the qualitative behaviour of the Greedy Gauss-Newton method and compare it with the $\ell_1$-method.
In Figure \ref{fig:TestQ1:MD} we show the results of the algorithm for solving
\eqref{eq:qp} using MD with  $n=6.$ In particular, we plot the absolute value of the solution $x$ obtained using MD, and the minus absolute value of the solution obtained using the $\ell_1$ - method in Figure \ref{fig:TestQ1:MD} (left upper).
The sparsity of the solution obtained by MD is equal to $n=6$ and the sparsity of the solution obtained by the $\ell_1$-method is $56.$ In Figure \ref{fig:TestQ1:MD} (right lower) one can see which columns of $J_k$ were added at each iteration step $k=1,2,...$.
We plot  $\|f(x_k)\|$ in logarithmic scale  in Figure \ref{fig:TestQ1:MD} (right upper) and the size of $\Omega_k$ in Figure \ref{fig:TestQ1:MD} (right lower) at each iteration.

The same test problem as in Figure \ref{fig:TestQ1:MD} is then solved  using OM. We display the results in  Figure
\ref{fig:TestQ1:OM}. Note that the solution with OM is not the same as the one for MD even if the sparsity is the same.

\begin{figure}[h!]
\centerline{
\includegraphics[width=1.1\textwidth]{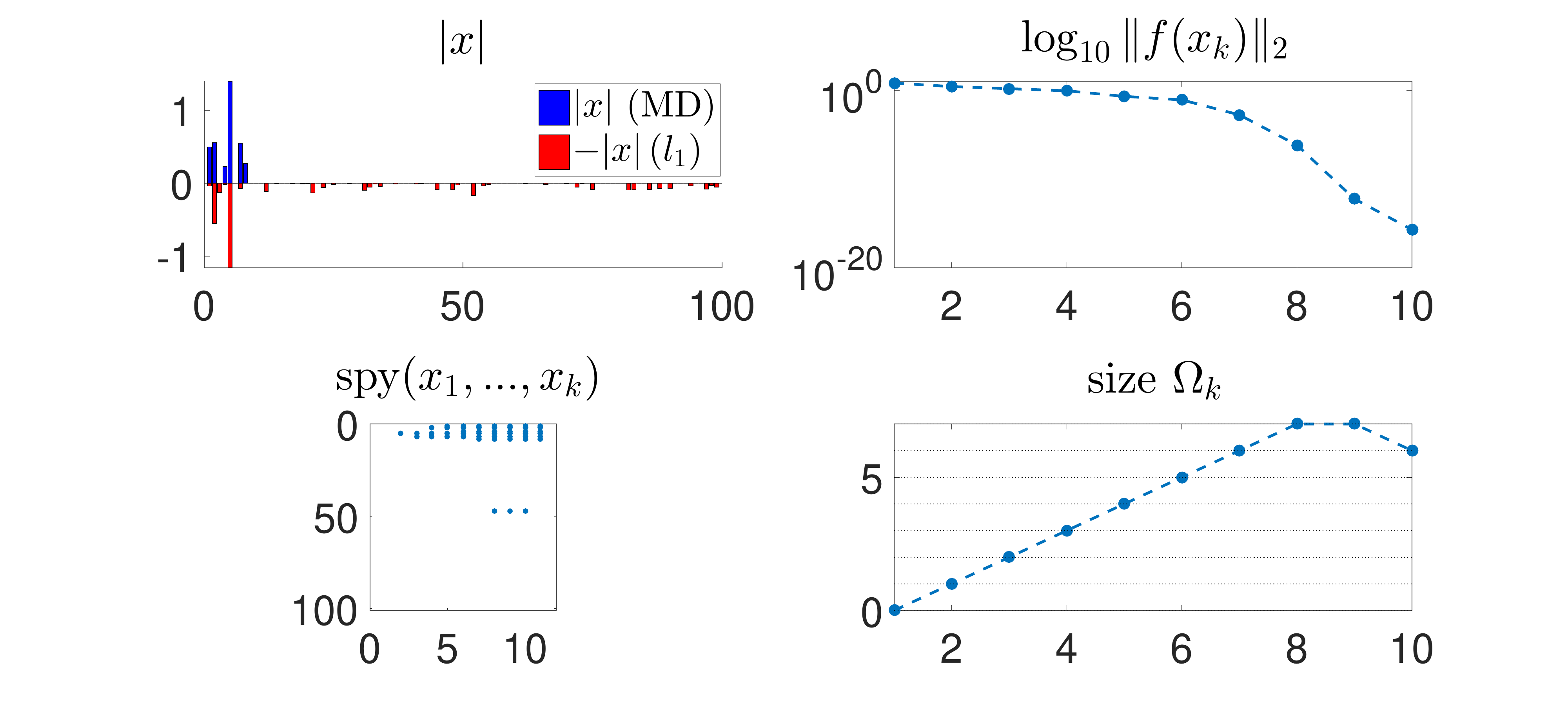}}
\caption{ MD method performance for the test problem \eqref{eq:qp} with $N=100,$ $m=20,$ $n=6$ and $s=2$}
\label{fig:TestQ1:MD}
\end{figure}

\begin{figure}[h!]
\centerline{
\includegraphics[width=1.1\textwidth]{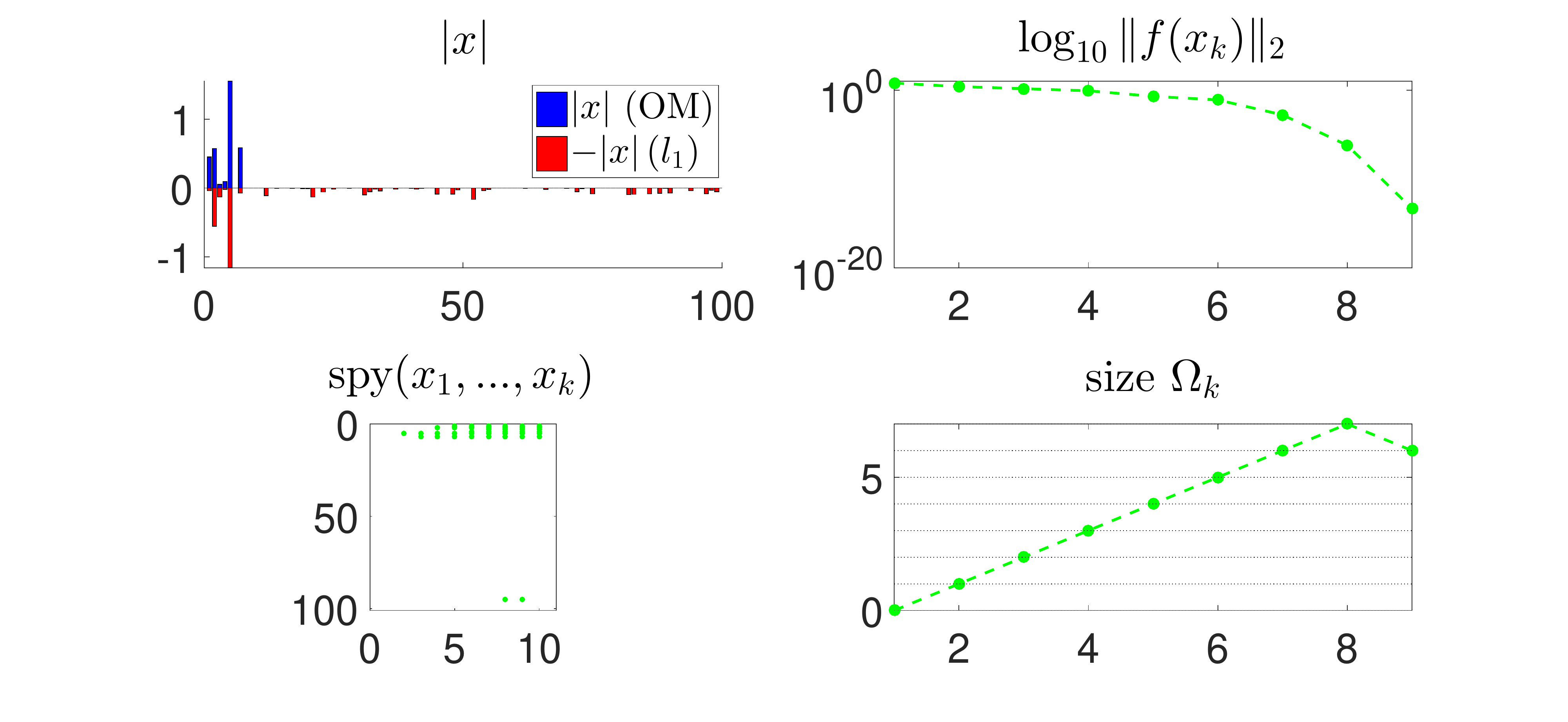}}
\caption{MD method performance for the test problem \eqref{eq:qp} with $N=100,$ $m=20,$ $n=6$ and $s=2.$}
\label{fig:TestQ1:OM}
\end{figure}

 For the chosen parameters the convergence to a sparse solution, as in Figure \ref{fig:TestQ1:MD} and Figure \ref{fig:TestQ1:OM}, is the most common case.  However, the algorithm may not produce a convergent (to the solution) sequence starting with $x_1=0$, see Figure \ref{fig:TestQ2}, or produce an m-sparse solution, as in Figure \ref{fig:TestQ3}.

 In Figure \ref{fig:TestQ2} (right upper), one can see an example of the case when the algorithm got stuck in a subspace with a local minimum to  $\|f(x)\|^2/2$ that does not yield a solution to $f(x)=0.$ The rank of the Jacobian at these minima are  equal to $18,19,19$ which can be seen from Figure \ref{fig:TestQ2} (right lower). The algorithm converged to a sparse solution after three (different) restarts.  We have plotted the absolute value of the solution and the minus absolute value of the solution of sparsity $56$ obtained by the $\ell_1$- method in Figure \ref{fig:TestQ2} (left upper). In Figure \ref{fig:TestQ2} (left lower) the subspace of the local minimum and the subspace of the solution are shown.

  Finally, in Figure \ref{fig:TestQ3} we show the case where the algorithm does not find a sparse solution but converges to a solution of the sparsity $m,$ $m=20.$ The sparsity of the solution obtained by $\ell_1$-method is equal to $54,$ see Figure \ref{fig:TestQ3} (lower left).

 Since we have not found significant difference in the qualitative behaviour between OM and MD we have displayed the results for the last two tests only for MD.

\begin{figure}[h!]
\centerline{
\includegraphics[width=1.1\textwidth]{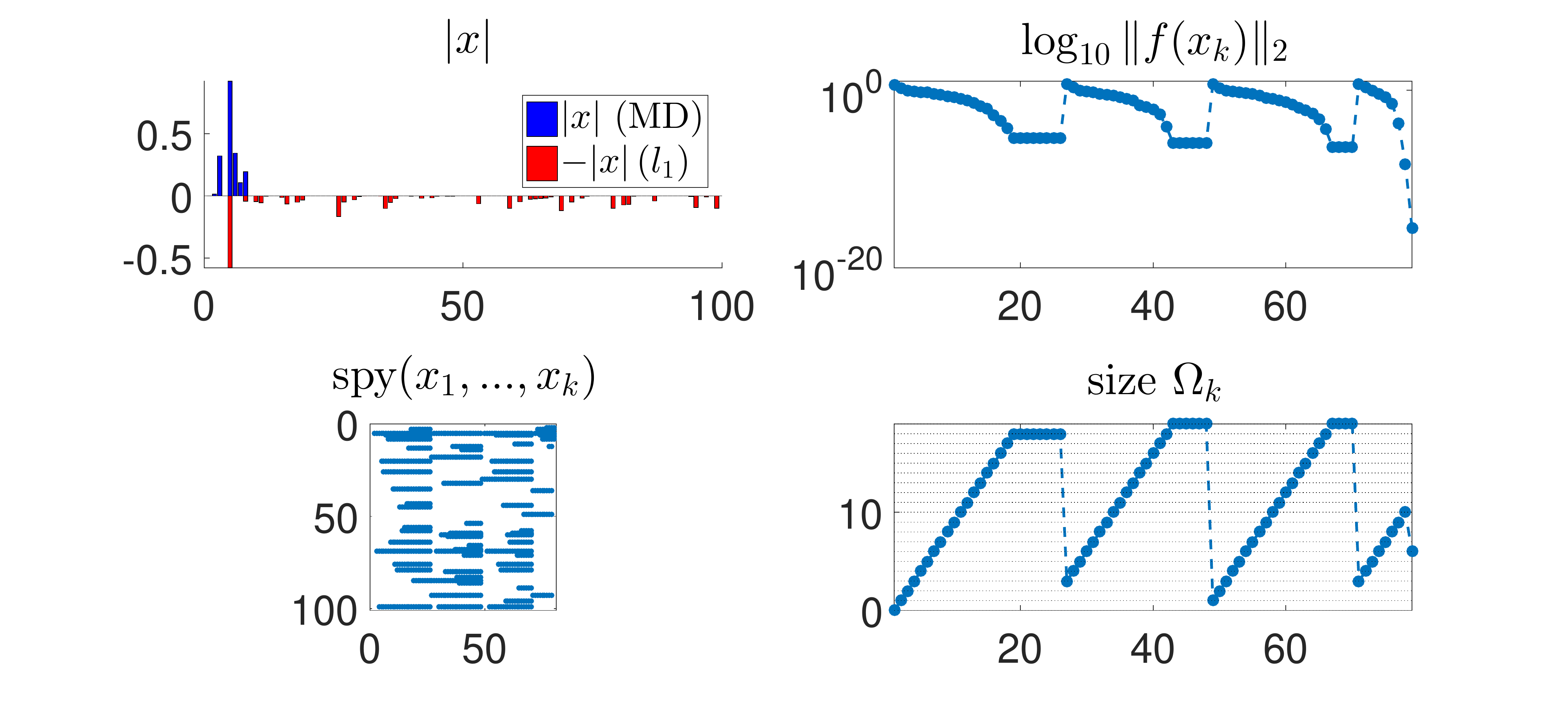}}
\caption{MD method performance for the test problem \eqref{eq:qp} with $N=100,$ $m=20,$ $n=6$ and $s=2$}
\label{fig:TestQ2}
\end{figure}

\begin{figure}[h!]
\centerline{
\includegraphics[width=1.1\textwidth]{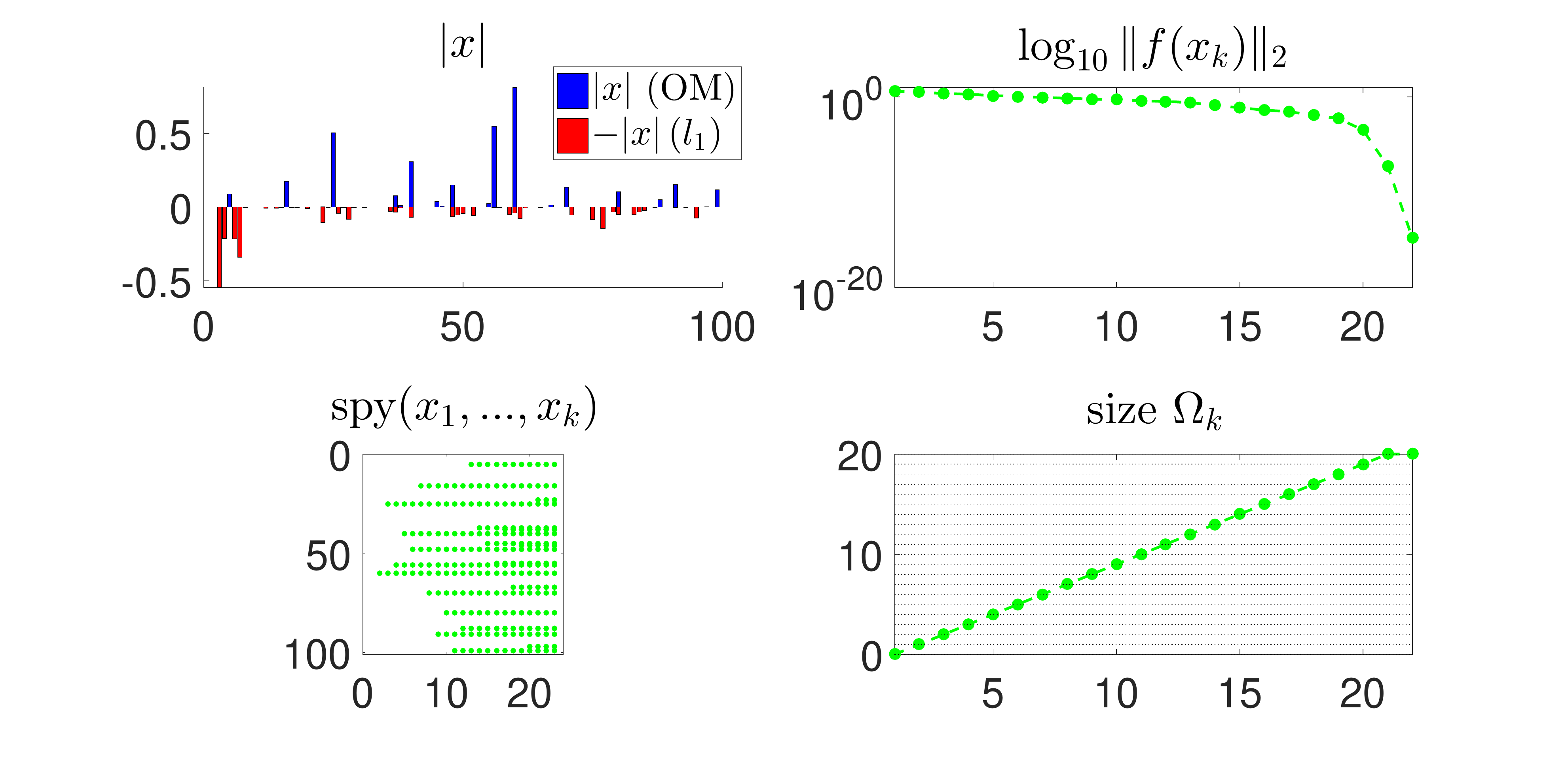}}
\caption{MD method performance for the test problem \eqref{eq:qp} with $N=100,$ $m=20,$ $n=6$ and $s=2$.}
\label{fig:TestQ3}
\end{figure}

We would like  note that while solutions obtained by the Greedy Gauss-Newton method can not exceed $m,$ the $\ell_1$-method
may produce a solution of even larger sparsity than $m$, which was the case for the considered test problem \eqref{eq:qp} for all our runs.

In Figure \ref{fig:LargeTest_MD_OM_3D} and  \ref{fig:LargeTest_MD_OM_proj} we illustrate the performance of the algorithm over the average of $10$ runs where $N=100$, $m$ and $n$ vary as $m=8,10,\ldots, 98$ and $n=2,4,\ldots,m-6.$.

The upper two plots in Figure \ref{fig:LargeTest_MD_OM_3D} show that the sparsity $n$ of the solution is attained except for a curved ridge. It has been shown  in \cite{Tropp2008} that for linear problems orthogonal matching pursuit can provably recover $n$-sparse signals when $n \leq m/(2\log (N))$. This estimate is illustrated by the cutting plane in the figures. It is seen that MD and OM manage to find less sparse solutions than the estimate. In the lower right plots in  Figure \ref{fig:LargeTest_MD_OM_3D} and  \ref{fig:LargeTest_MD_OM_proj} it is seen that MD outperforms OM for most problem sizes. The number of restarts were insignificantly small for these tests.

\begin{figure}[h!]
\centerline{
\includegraphics[width=1.1\textwidth]{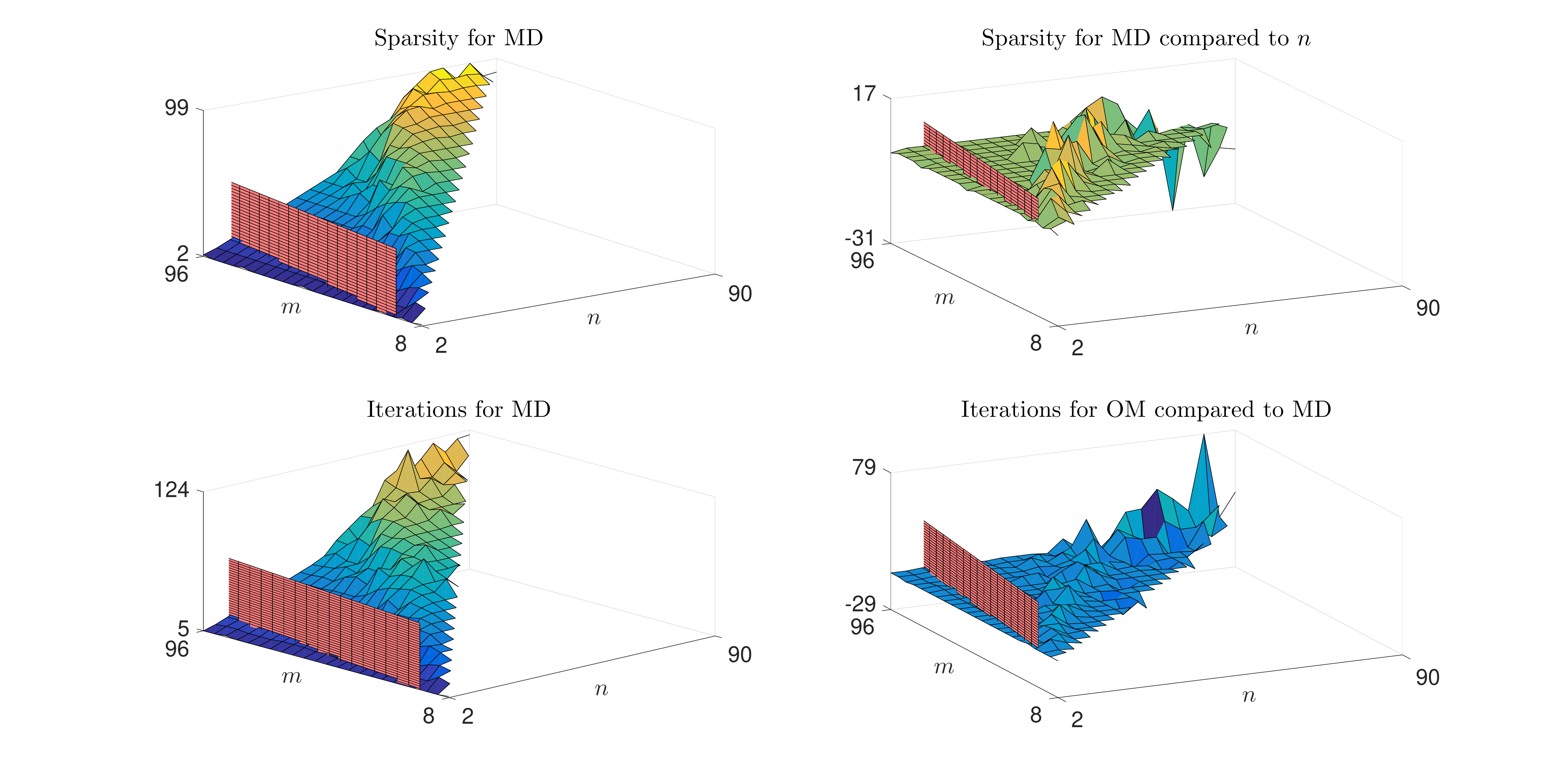}}
\caption{The 3D plots of the performance of MD and OM methods  for the test problem \eqref{eq:qp} over the average of $10$ runs, $N=100,$ $s=2,$ and $ m=8,10, \ldots, 98$, $n=2,4, \ldots,m-6$.}
\label{fig:LargeTest_MD_OM_3D}
\end{figure}

\begin{figure}[h!]
\centerline{
\includegraphics[width=1.1\textwidth]{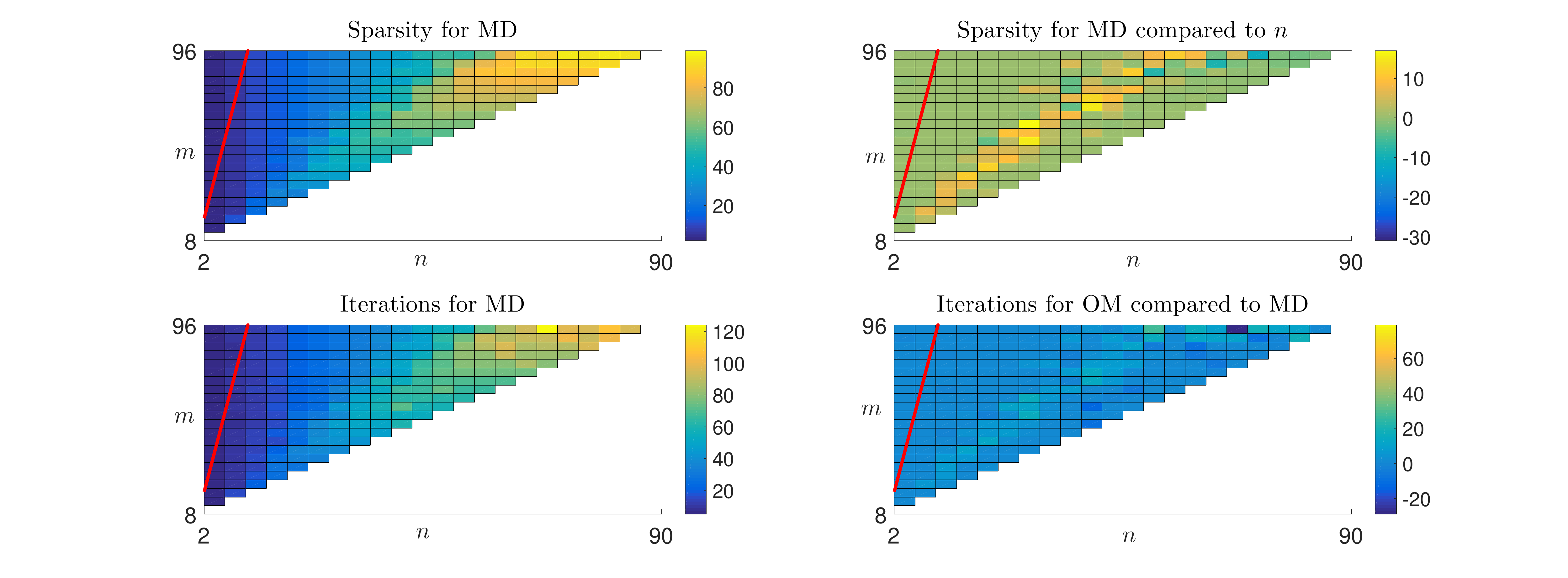}}
\caption{The contour plots of the performance of MD and OM methods  for the test problem \eqref{eq:qp} over the average of $10$ runs, $N=100,$ $s=2,$ and $ m=8,10, \ldots, 98$, $n=2,4, \ldots,m-6$.}
\label{fig:LargeTest_MD_OM_proj}
\end{figure}
%
%

\subsection{Exponential Test Problem}

This problem is taken from \cite{Ma89}.

Define

\begin{equation}\label{eq:ep}
   f(x) = Ae^{Bx}-b,\, A \in \mathbb{R}^{m\times N}, \, b\in \mathbb{R}^{m},\, e^x = \left(e^{x_1}, \ldots, e^{x_N}\right)^T
\end{equation}
where the elements in $A$ are chosen random uniformly in $(-1,1)$ and then by using Singular Value Decomposition to have $\rank (A)=m-p$. The matrix $B$ is constructed in the following way. First, we generate $N \times N$ random matrix whose elements are uniformly distributed in $[-1,1].$ Next, using Singular Value Decomposition we fix this matrix to have the first $n+s<m$ columns to have the rank $n$ for some $n,s \in \N$.
That is, $\rank (B(:,1:n+s))=n$ and $B$ most probably has the rank $N-s.$

We choose $\bar{x}=(\bar{z},0)^T$ with some $\bar{z}\in \R^{n+s}$ and set $b=A \exp(B\bar{x}).$
Then for any $y\in \R^s$
\begin{equation} \label{eq:x:ep}
x=\bar{x}+\begin{pmatrix} V_2y\\ 0\end{pmatrix}
\end{equation}
solves $f(x)=0$ with $V_2\in \R^{(n+s)\times s}$ such that $\cR(V_2)= \cN(B(:,1:n+s)).$ From this construction it is clear that some of $x$ among \eqref{eq:x:ep} have the sparsity $n.$

The Jacobian and second derivatives are given as
\[
   J(x) = A \, \text{diag}\left(e^{x_1}, \ldots , e^{x_N}\right) = A \, \text{diag}(e^x), \quad  f''_i = \text{diag}\left(a_{i1}e^x_1, \ldots , a_{iN}e^{x_N}\right), \quad i=1,...,m.
\]
where $a_{ij},$ $j=1,...,N,$ are the elements of $A.$

The matrix $J(x)$ is always rank deficient.  Indeed, since $ A \, \text{diag}\left(e^{x_1}, \ldots , e^{x_N}\right) $ has the same rank as $A$ we have
\begin{equation*}
\begin{split}
   \rank(J(x))  \leq \min \left\{ \rank(A), \rank(B) \right\} = \min  \left\{ m-p, N-s \right\}
\end{split}
\end{equation*}

\begin{figure}[h!]
\centerline{
\includegraphics[width=1.1\textwidth]{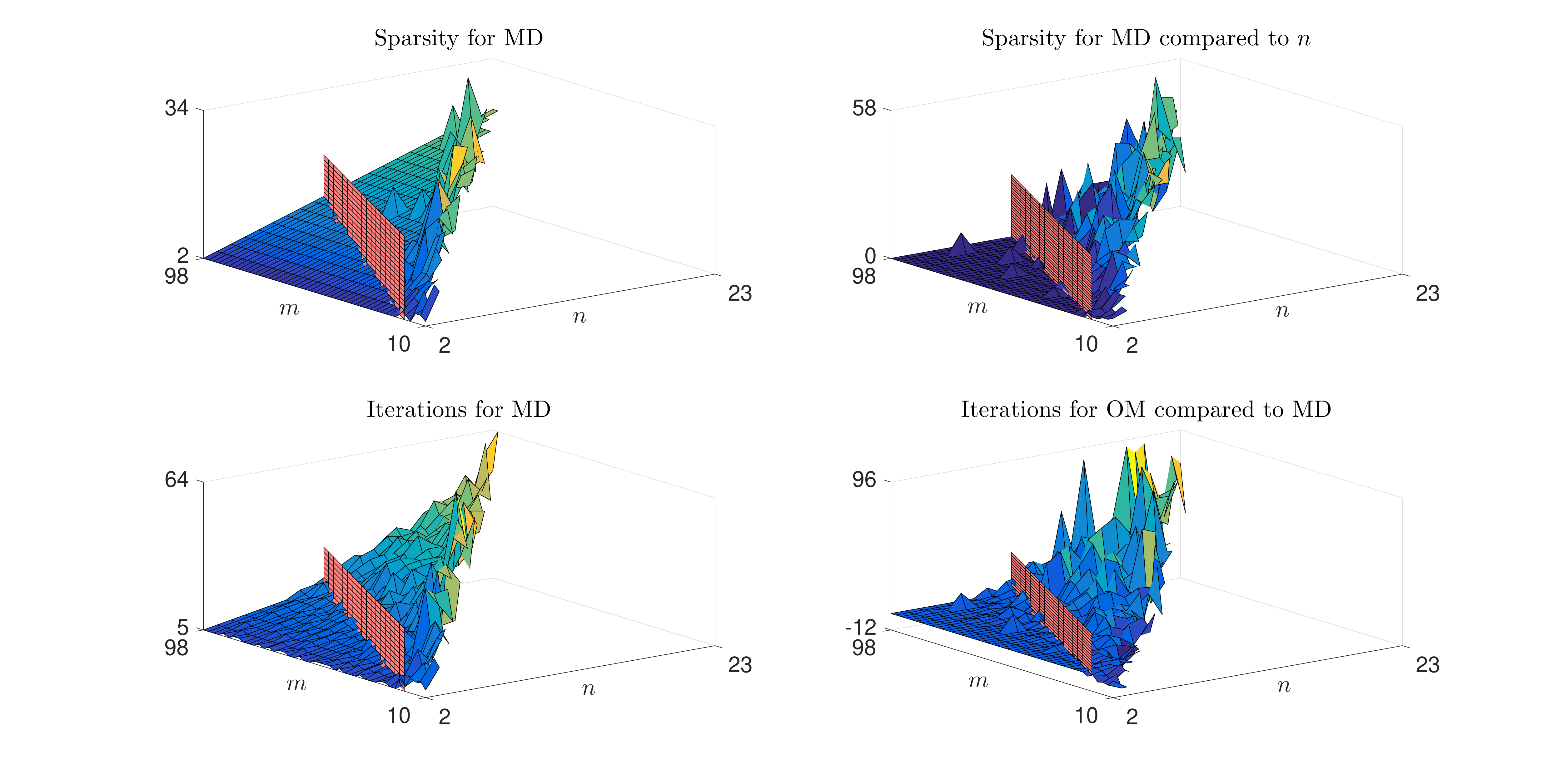}}
\caption{
The performance of MD and OM methods for the test problem with \eqref{eq:ep} over the average of $10$ runs, $N=100,$ $s=4,$ and $ m=12,16, \ldots, 96$, $n=2,6, \ldots,m-10$.}
\label{fig:LargeTest_Exp_3D}
\end{figure}

\begin{figure}[h!]
\centerline{
\includegraphics[width=1.1\textwidth]{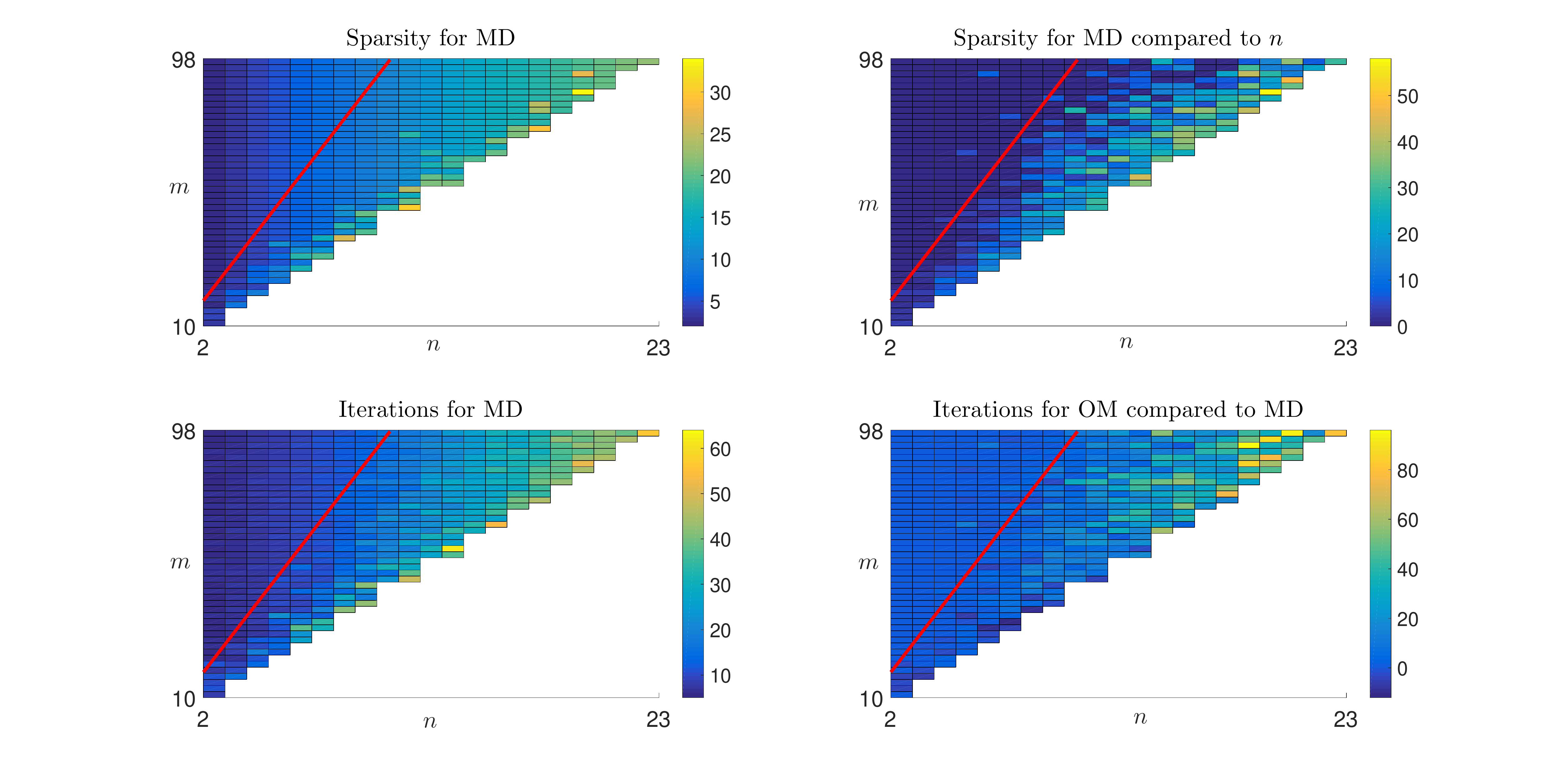}}
\caption{The performance of MD and OM methods for the test problem with \eqref{eq:ep} over the average of $10$ runs, $N=100,$ $s=4,$ and $ m=12,16, \ldots, 96$, $n=2,6, \ldots,m-10$.}
\label{fig:LargeTest_Exp_proj}
\end{figure}

%

All the tests were run with $N=100$, $s=2,$   $prob = (2+m/10)/100$ and the constants given in see Section \ref{Sec:Pseudo}. Furthermore, an additional condition for a restart, $\max_i | x_k(i) | >10^3$, is added in the condition of the if-statement on row 11 in the pseudocode to prevent convergence to infinity.

In Figure \ref{fig:LargeTest_Exp_3D} and  \ref{fig:LargeTest_Exp_proj} we illustrate the performance of the algorithm over the average of $10$ runs where $N=100, s=4$, $m, n$ vary as $ m=12,16, \ldots, 96$, $n=2,6, \ldots,m-10$.

The upper right  plots in Figure \ref{fig:LargeTest_MD_OM_3D} and \ref{fig:LargeTest_Exp_proj} show that the sparsity of the solution is attained very close to the estimate  $n \leq m/(2\log (N))$ obtained for linear problems. We however do not have theoretical justification of this estimate for nonlinear cases.  Figure \ref{fig:LargeTest_MD_OM_3D} (lower right) and \ref{fig:LargeTest_Exp_proj} (lower right) shows that MD outperforms OM for all problem sizes. The number of restarts for this test problem were more frequent than for the quadratic test problem, see Section \ref{Sec:Quadratic}. However, there were few cases when $m\approx n$ and $m$ is large, where there was no convergence.

\vspace{3cm}

\bibliographystyle{unsrt}

\end{document}